\documentclass[11pt,a4paper]{article}
\usepackage{mathrsfs}
\usepackage{epsf,epsfig,amsfonts,amsgen,amsmath,amstext,amsbsy,amsopn,amsthm,lineno}
\usepackage{color}
\newtheorem{theorem}{Theorem}

\newtheorem{lemma}{Lemma}
\newtheorem{false statement}{False statement}

\theoremstyle{definition}

\newtheorem{claim}{Claim}

\newtheorem{case}{Case}
\newtheorem{subcase}{Case}[case]
\newtheorem{subsubcase}{Case}[subcase]

\begin{document}
\title{\bf\Large Bipartite graphs with the maximum sum of squares of degrees\thanks{Supported  by NSFC
(No.~10871158).}}

\date{August 9, 2009}

\author{Shenggui Zhang\thanks{Corresponding author. E-mail address: sgzhang@nwpu.edu.cn (S. Zhang).} and Chuncao
Zhou\\[2mm]
\small Department of Applied Mathematics,
\small Northwestern Polytechnical University,\\
\small Xi'an, Shaanxi 710072, P.R.~China\\}
\maketitle

\begin{abstract}
In this paper we determine all the bipartite graphs with the maximum
sum of squares of degrees among the ones with a given number of
vertices and edges.

\smallskip

\end{abstract}

\section{Introduction}

All graphs considered here are finite, undirected and simple. For
terminology and notation not defined here we follow those in Bondy
and Murty \cite{Bondy_Murty}.

In this paper we study an extremal problem on bipartite graphs:
among all bipartite graphs with a given number of vertices and
edges, find the ones where the sum of squares of degrees is maximum.

The corresponding problem for general graphs has been studied in
\cite{Ahlswede_Kanota,Boesch,Peled_Petreschi_Sterbini}. For all
graphs with a given number vertices and edges, Ahlswede and Kanota
\cite{Ahlswede_Kanota} first determined the maximum sum of squares
of degrees. Boesch et al. \cite{Boesch} proved that if the sum of
squares of degrees attains the maximum, then the graph must be a
threshold graph (See the definition in \cite{Mahadev_Peled}). They
constructed two threshold graphs and proved that at least one of
them is such an extremal graph. Peled et al.
\cite{Peled_Petreschi_Sterbini} further studied this problem and
showed that, if a graph has the maximum sum of squares of degrees,
then it must belong to one of the six particular classes of
threshold graphs.

For the family of bipartite graphs with a given number of vertices
and edges and the size of one partite side, Ahlswede and Kanota
\cite{Ahlswede_Kanota} determined a bipartite graph such that the
sum of squares of its degrees is maximum. Recently, Cheng et al.
\cite{Cheng_Guo_Zhang_Du} determined the maximum sum of squares of
degrees for bipartite graphs with a given number of vertices and
edges.

While the problem of finding all the graphs with a given number of
vertices and edges where the sum of squares of degrees is maximum is
still unsolved, we give a complete solution to the problem of
finding all the bipartite graphs with a given number of vertices and
edges where the sum of squares of degrees is maximum in this paper.
In Section 2 we present some notation and lemmas that will be used
later and in Section 3 give the main results and the proof.

\section{Notation and lemmas}

Let $x$ be a real number. We use $\lfloor x\rfloor$ to represent the
largest integer not greater than $x$ and $\lceil x\rceil$ to
represent the smallest integer not less than $x$. The sign of $x$,
denoted by $sgn(x)$, is defined as $1$, $-1$, and $0$ when $x$ is
positive, negative and zero, respectively.

Let $n$, $m$ and $k$ be three positive integers. We use $B(n,m)$ to
denote a bipartite graph with $n$ vertices and $m$ edges, and
$B(n,m,k)$ to denote a $B(n,m)$ with a bipartition $(X,Y)$ such that
$|X|=k$. By $\mathscr{B}(n,m)$ we denote the set of graphs of the
form $B(n,m)$ and $\mathscr{B}(n,m,k)$ the set of graphs of the form
$B(n,m,k)$.

Suppose that $n$, $m$ and $k$ are three integers with $n\geq 2$,
$0\leq m\leq \lfloor \frac{n}{2}\rfloor\lceil \frac{n}{2}\rceil$ and
$1\leq k\leq n-1$. Let $m=qk+r$, where $0\leq r<k$. Then
$B^l(n,m,k)$ is defined as a bipartite graph in $\mathscr{B}(n,m,k)$
such that $q$ vertices in $Y$ are adjacent to all the vertices of
$X$ and one more vertex in $Y$ is adjacent to $r$ vertices in $X$ if
$r>0$.

We use $\mathscr{G}(n,m)$ to denote the family of graphs with $n$
vertices and $m$ edges. Given an integer $t\geq 2$, and a graph
$G\in\mathscr{G}(n,m)$, let
$$
\sigma_t(G)=\sum_{v\in V(G)}(d(v))^t.
$$

The following result is due to Ahlswede and Kanota
\cite{Ahlswede_Kanota}.

\begin{lemma}[Ahlswede and Kanota \cite{Ahlswede_Kanota}]
Let $n,m$ and $k$ be three integers with $n\geq 2$,
$0\leq{m}\leq\lfloor\frac{n}{2}\rfloor\lceil \frac{n}{2}\rceil$ and
$\lceil\frac{n}{2}\rceil\leq k\leq n-1$. Suppose that $m=qk+r$,
where $0\leq{r}<k$. Then $\sigma_2(B^l(n,m,k))$ attains the maximum
value among all the graphs in $\mathscr{B}(n,m,k)$.
\end{lemma}

With this result, Cheng et al. \cite{Cheng_Guo_Zhang_Du} obtained
the following

\begin{lemma}[Cheng, Guo, Zhang and Du \cite{Cheng_Guo_Zhang_Du}]
Let $n,m$ be two integers with $n\geq 2$, $n\leq{m}\leq
\lfloor\frac{n}{2}\rfloor\lceil \frac{n}{2}\rceil$ and
$k_0=max\{k|m=qk+r,0\leq r<k, \lceil \frac{n}{2}\rceil\leq{k}\leq
n-q-sgn(r)\}$.  Then $\sigma_2(B^l(n,m,k_0))$ attains the maximum
value among all the bipartite graphs in $\mathscr{B}(n,m)$.
\end{lemma}

For general graphs with few edges, Ismailescu and Stefanica
\cite{Ismailescu_Stefanica} got the following result.

\begin{lemma}[Ismailescu and Stefanica \cite{Ismailescu_Stefanica}]
Let $n,m$ and $t$ be three integers with $n\geq 2$, $m\leq n-2$ and
$t\geq 2$. Suppose that $\sigma_t(G^\ast)$ attains the maximum value
among all the graphs in $\mathscr{G}(n,m)$. Then $G^\ast\cong
K_{1,m}\cup S_{n-m-1}$, the star with $m$ edges plus $n-m-1$
isolated vertices, except the case $t=2$ and $m=3$, where both
$\sigma_t(K_{1,3}\cup S_{n-4})$ and $\sigma_t(K_3\cup S_{n-3})$
attains the maximum.
\end{lemma}

Let $B$ be a bipartite graph. We use $\overline B$ to denote the
bipartite graph on the same partition as $B$ such that two vertices
in $\overline B$ are adjacent if and only if they are not adjacent
in $B$.

\begin{lemma}
{Let $B$ be a bipartite graph in $\mathscr{B}(n,m,k)$. Then
$\sigma_2(B)$ attains the maximum value among all the graphs in
$\mathscr{B}(n,m,k)$ if and only if $\sigma_2(\overline B)$ attains
the maximum value among all the graphs in
$\mathscr{B}(n,k(n-k)-m,k)$.}
\end{lemma}

\begin{proof}
Let $(X,Y)$ be the bipartition of $B$. Suppose that
$X=\{x_1,x_2,\dots,x_{k}\}$ and $Y=\{y_1,y_2,\dots,y_{n-k}\}$.
Denote the degree of $x_i$ in $\overline B$ by $\overline d(x_i)$
for $i=1,2,\ldots,k$ and the degree of $y_j$ in $\overline B$ by
$\overline d(y_j)$ for $j=1,2,\ldots,n-k$. Then we have
$$
d(x_i)+\overline d(x_i)=n-k \mbox{\ \ for $i=1,2,\ldots,k$},
d(y_j)+\overline d(y_j)=k \mbox{\ \ for $j=1,2,\ldots,n-k$},
$$
and
$$
\sum^{k}_{i=1}\overline d(x_i)=\sum^{n-k}_{j=1}\overline
d(y_j)=k(n-k)-m.
$$
Therefore,
\begin{eqnarray*}
\sigma_2(B) &=& \sum^{k}_{i=1}d(x_i)^2+\sum^{n-k}_{j=1}d(y_j)^2\\
                     &=& \sum^{k}_{i=1}(n-k-\overline d(x_i))^2+\sum^{n-k}_{j=1}(k-\overline d(y_j))^2\\
                     &=& k(n-k)^2-2(n-k)\sum^{k}_{i=1}\overline d(x_i)+\sum^{k}_{i=1}\overline d(x_i)^2\\
                     & & +(n-k)k^2-2k\sum^{n-k}_{j=1}\overline d(y_j)+\sum^{n-k}_{j=1}\overline d(y_j)^2\\
                     &=& n(2m+k^2-nk)+\sum^{k}_{i=1}\overline d(x_i)^2+\sum^{n-k}_{j=1}\overline d(y_j)^2\\
                     &=& n(2m+k^2-nk)+\sigma_2(\overline B).
\end{eqnarray*}
The result follows immediately.
\end{proof}

\section{Main results}

We first determine the bipartite graphs with few edges where the sum
of squares of degrees is maximum.

\begin{theorem}
Let $n,m$ be two integers with $n\geq 2$ and $0\leq{m}\leq{n-1}$.
Suppose that $\sigma_2(B^\ast)$ attains the maximum value among all
the graphs in $\mathscr{B}(n,m)$. Then $B^\ast\cong K_{1,m}\cup
S_{n-m-1}$.
\end{theorem}

\begin{proof}
From Lemma 2 we know that $\sigma_2(B^l(n,m,k_0))$ attains the
maximum value among all the bipartite graphs in $\mathscr{B}(n,m)$,
where $k_0=max\{k|m=qk+r,0\leq r<k, \lceil
\frac{n}{2}\rceil\leq{k}\leq n-q-sgn(r)\}$. So we have
$\sigma_2(B^l(n,m,k_0))=\sigma_2(B^*)$. We distinguish two cases.

\begin{case}
$0\leq{m}\leq{n-2}$.
\end{case}

Let $m=q_0k_0+r_0$, where $0\leq r_0<k_0$. Then we can conclude
$k_0=n-1$, $q_0=0$ and $r_0=m$. Hence, $B^l(n,m,k_0)=K_{1,m}\cup
S_{n-m-1}$. By Lemma 3 we know that $K_{1,m}\cup S_{n-m-1}$ is the
unique bipartite graph with the maximum sum of squares of degrees in
$\mathscr{B}(n,m)$. So we have $B^\ast\cong K_{1,m}\cup S_{n-m-1}$.

\begin{case}
$m=n-1$.
\end{case}

In this case we have $B^l(n,m,k_0)=K_{1,n-1}$. Therefore,
$\sigma_2(K_{1,n-1})=\sigma_2(B^*)$. If $B^*\not\cong K_{1,n-1}$,
then
$$
\sigma_2(K_{1,n-1}\cup
S_1)=\sigma_2(K_{1,n-1})=\sigma_2(B^*)=\sigma_2(B^*\cup S_1),
$$
which is a contradiction to the result in the Case 1.
\end{proof}

\begin{theorem}
Let $n,m$ be two integers with $n\geq 2$, $n\leq{m}\leq \lfloor
\frac{n}{2}\rfloor\lceil \frac{n}{2}\rceil$ and
$k_0=max\{k|m=qk+r,0\leq r<k, \lceil \frac{n}{2}\rceil\leq{k}\leq
n-q-sgn(r)\}$. Suppose that $\sigma_2(B^\ast)$ attains the maximum
value among all the graphs in
$\mathscr{B}(n,m)$. Then\\
$(a)$ $B^\ast\cong B^l(n,m,k_0)$, or $B^l(n,m,n-k_0)$ if $m>(n-k_0)(k_0-1)$;\\
$(b)$ $B^\ast\cong B^l(n,m,k_0)$, $B^l(n,m,n-k_0)$, or $B^l(n,m,k_0-1)$ if $m=(n-k_0)(k_0-1)$;\\
$(c)$ $B^\ast\cong B^l(n,m,k_0)$ if $m<(n-k_0)(k_0-1)$.
\end{theorem}

\begin{proof}

Let $m=q_0k_0+r_0=q'_0(k_0+1)+r'_0$, where  $0\leq r_0<k_0$, $0\leq
r'_0<k_0+1$. We first prepare three claims.

\begin{claim}
$m>(k_0+1)(n-k_0-1)$.
\end{claim}

\begin{proof}
Suppose that $m\leq (k_0+1)(n-k_0-1)$. Then $B^l(n,m,k_0+1)$ exists
in $\mathscr{B}(n,m,k_0+1)$. This implies that $k_0+1\leq
n-q'_0-sgn(r'_0)$, contradicting the maximum of $k_0$.
\end{proof}

\begin{claim}
{There exist no isolated vertices in $B^l(n,m,k_0)$.}
\end{claim}

\begin{proof}
Suppose that there exists an isolated vertex in $B^l(n,m,k_0)$.
Since $n\leq m$, we have $q_0\geq 1$. Let $(X_0,Y_0)$ be the
bipartition of $B^l(n,m,k_0)$ with $|X_0|=k_0$. Then by the
definition of $B^l(n,m,k_0)$, the isolated vertex must be in $Y_0$.
Hence we have $m\leq k_0(n-k_0-1)\leq (k_0+1)(n-k_0-1)$,
contradicting Claim 1.
\end{proof}

Let $k\geq \lceil \frac{n}{2}\rceil$ be an integer. Suppose that
$m=qk+r=q'(k+1)+r'$, where $0\leq r<k$, $0\leq r'<k+1$. Then we have
$q=\lfloor\frac{m}{k}\rfloor$ and $q'=\lfloor\frac{m}{k+1}\rfloor$.

\begin{claim}
$\lfloor\frac{m}{k}\rfloor-\lfloor\frac{m}{k+1}\rfloor\leq 1$.
\end{claim}

\begin{proof}
If $\lfloor\frac{m}{k}\rfloor-\lfloor\frac{m}{k+1}\rfloor\geq 2$,
then
\begin{eqnarray}
\nonumber r^\prime &=& \lfloor \frac{m}{k}\rfloor k+r-\lfloor\frac{m}{k+1}\rfloor(k+1)\\
\nonumber          &\geq& \lfloor \frac{m}{k}\rfloor k+r-(\lfloor\frac{m}{k}\rfloor-2)(k+1)\\
\nonumber          &=& r+2(k+1)- \lfloor \frac{m}{k}\rfloor\\
\nonumber          &\geq& r+2(k+1)-k\\
\nonumber          &>& k+1,
\end{eqnarray}
a contradiction.
\end{proof}

By the definition of $B^l(n,m,k)$, we have
\begin{eqnarray}
\nonumber \sigma_2(B^l(n,m,k)) &=& r(q+1)^2+(k-r)q^2+qk^2+r^2\\
\nonumber                      &=& (m-qk)(q+1)^2+(k+qk-m)q^2+qk^2+(m-qk)^2\\
\nonumber                      &=& q(k-1)(k+qk-2m)+m^2+m\\
\nonumber                      &=&\lfloor\frac{m}{k}\rfloor(k-1)
(k+\lfloor\frac{m}{k}\rfloor k-2m)+m^2+m.
\end{eqnarray}
Set $f(k)=\sigma_2(B^l(n,m,k))$. Then
$$ f(k+1)-f(k)=\lfloor\frac{m}{k+1}\rfloor k(k+1+\lfloor \frac{m}{k+1}\rfloor
(k+1)-2m)-\lfloor\frac{m}{k}\rfloor(k-1)(k+\lfloor\frac{m}{k}\rfloor
k-2m).$$ If
$\lfloor\frac{m}{k}\rfloor-\lfloor\frac{m}{k+1}\rfloor=0$, then
\begin{equation}\label{1}
f(k+1)-f(k)=2\lfloor\frac{m}{k}\rfloor(\lfloor\frac{m}{k}\rfloor
k+k-m)>0.
\end{equation}
If $\lfloor\frac{m}{k}\rfloor-\lfloor \frac{m}{k+1}\rfloor=1$, then
\begin{equation}\label{2}
f(k+1)-f(k)=2(\lfloor\frac{m}{k}\rfloor-k)(\lfloor\frac{m}{k}\rfloor
k-m)\geq 0.
\end{equation}
Thus, $f(k)$ is a nondecreasing function. So we have
\begin{equation}\label{3}
f(k_0)\geq f(k_0-1)\geq f(k_0-2)\geq\dots\geq f(\lceil
\frac{n}{2}\rceil).
\end{equation}

By Lemma 1, we know that
$\sigma_2(B^*)=\max\{f(k_0),f(k_0-1),\ldots,f(\lceil\frac{n}{2}\rceil)\}$.
Let $(X^\ast,Y^\ast)$ be the bipartition of $B^\ast$ with
$|X^\ast|\geq \lceil n/2\rceil$. We distinguish two cases.

\setcounter{case}{0}
\begin{case}
$k_0=\lceil \frac{n}{2}\rceil$.
\end{case}

First, we have $n=2k_0$ or $2k_0-1$. It is clear that
\begin{equation}\label{4}
m\leq k_0(n-k_0).
\end{equation}

Suppose that $n=2k_0$. Then by Claim 1 and (4) we have
$$
k_0^2-1<m\leq k_0^2,
$$
i.e., $m=k_0^2$. This means that $B^l(n,m,k_0)$ is the
unique graph in $\mathscr{B}(n,m)$. So we have $B^*\cong
B^l(n,m,k_0)$.

Suppose that $n=2k_0-1$. Then by Claim 1 and (4) we have
$$
(k_0+1)(k_0-2)<m\leq k_0(k_0-1).
$$
This implies that $m=k_0(k_0-1)$ or $k_0(k_0-1)-1$. In either cases,
$B^l(n,m,k_0)$ is the unique graph in $\mathscr{B}(n,m)$. So we have
$B^*\cong B^l(n,m,k_0)$.

\begin{case}
$k_0>\lceil \frac{n}{2}\rceil$.
\end{case}

\begin{subcase}
$f(k_0)=f(k_0-1)$.
\end{subcase}

Let $m=q_0''(k_0-1)+r_0''=q_0'''(k_0-2)+r_0'''$, where $0\leq
r_0''<k_0-1$, $0\leq r_0'''<k_0-2$. Then we have
$q_0''=\lfloor\frac{m}{k_0-1}\rfloor$ and
$q_0'''=\lfloor\frac{m}{k_0-2}\rfloor$.

Since $f(k_0)=f(k_0-1)$, it follows from (1) and (2) that
$$
f(k_0)-f(k_0-1)=2(\lfloor\frac{m}{k_0-1}\rfloor-(k_0-1))(\lfloor
\frac{m}{k_0-1}\rfloor (k_0-1)-m)=0.
$$
So we have
$$
\lfloor\frac{m}{k_0-1}\rfloor-(k_0-1)=0 \mbox{\ \ or\ \ }
\lfloor\frac{m}{k_0-1}\rfloor (k_0-1)-m=0.
$$

Suppose that $\lfloor\frac{m}{k_0-1}\rfloor-(k_0-1)=0$. Since
$k_0-1\geq\lceil\frac{n}{2}\rceil$, we have
\begin{eqnarray*}
m \geq (k_0-1)^2\geq (\lceil\frac{n}{2}\rceil)^2.
\end{eqnarray*}
By the condition
$m\leq\lfloor\frac{n}{2}\rfloor\lceil\frac{n}{2}\rceil$, we can
easily deduce that $m=(k_0-1)^2$. Again, with
$k_0-1\geq\lceil\frac{n}{2}\rceil$, we have
\begin{eqnarray*}
m=(k_0-1)^2>k_0(k_0-2)\geq k_0(n-k_0),
\end{eqnarray*}
a contradiction.

Suppose that $\lfloor\frac{m}{k_0-1}\rfloor(k_0-1)-m=0$.
Then we have $r_0''=0$. Since $f(k_0)=f(k_0-1)$, by (1) and (2)
we can conclude that
$\lfloor\frac{m}{k_0-1}\rfloor-\lfloor\frac{m}{k_0}\rfloor=1$.

Suppose that $r_0=0$. Then
$$
m=q_0k_0=\lfloor\frac{m}{k_0-1}\rfloor(k_0-1)=(q_0+1)(k_0-1).
$$
This implies that $q_0=k_0-1$. It follows from Claim 2 that
$k_0=\lceil\frac{n}{2}\rceil$, a contradiction.

Suppose $r_0\neq0$. Then by Claim 2, we can conclude that
$k_0+q_0+1=n$. So we have
\begin{equation}\label{5}
m=\lfloor\frac{m}{k_0-1}\rfloor(k_0-1)=(\lfloor\frac{m}{k_0}\rfloor+1)(k_0-1)=(n-k_0)(k_0-1).
\end{equation}

Suppose that $k_0-2\geq \lceil \frac{n}{2}\rceil$ and
$f(k_0)=f(k_0-1)=f(k_0-2)$. Then it follows from (1) and (2) that
$$f(k_0-1)-f(k_0-2)=2(\lfloor\frac{m}{k_0-2}\rfloor-(k_0-2))(\lfloor\frac{m}{k_0-2}\rfloor(k_0-2)-m)=0.$$
As the proof of $\lfloor\frac{m}{k_0-1}\rfloor-(k_0-1)\neq 0$ for the case
 $f(k_0)=f(k_0-1)$, we can prove that
$\lfloor\frac{m}{k_0-2}\rfloor-(k_0-2)\neq0$. So let us now assume
that $\lfloor \frac{m}{k_0-2}\rfloor (k_0-2)-m=0$. Then we have
$r_0'''=0$. Since $f(k_0-1)=f(k_0-2)$, by (1) and (2) we can
conclude that
$\lfloor\frac{m}{k_0-2}\rfloor-\lfloor\frac{m}{k_0-1}\rfloor=1$.
Then, by (5), we have
$$
m=(n-k_0)(k_0-1)=\lfloor\frac{m}{k_0-2}\rfloor
(k_0-2)=(n-k_0+1)(k_0-2).
$$
This implies that $n=2k_0-2$, contradicting our assumption
$k_0-2\geq \lceil \frac{n}{2}\rceil$.

Therefore, we have $f(k_0)=f(k_0-1)>f(k_0-2)$. This means that  $B^*\in \mathscr{B}(n,m,k_0)$ or
$\mathscr{B}(n,m,k_0-1)$.

Suppose that $B^*\in \mathscr{B}(n,m,k_0)$.  Then $\sigma_2(B^*)$
attains the maximum value among all the graphs in
$\mathscr{B}(n,m,k_0)$. Note that $m=(n-k_0)(k_0-1)$. So we have
$k_0(n-k_0)-m=n-k_0$. It follows from Lemma 4 that
$\sigma_2(\overline {B^*})$ attains the maximum value among all the
graphs in $\mathscr{B}(n,n-k_0,k_0)$. By Theorem 1, we obtain that
$\overline {B^*}\cong K_{1,n-k_0}\cup S_{k_0-1}$. If the $n-k_0$
pendent vertices of $\overline {B^*}$ are in $X^\ast$, then by Lemma 4, we have
$B^\ast\cong B^l(n,m,k_0)$. If the $n-k_0$ pendent vertices of
$\overline {B^*}$ are in $Y^\ast$, then by Lemma 4, we have $B^\ast\cong
B^l(n,m,n-k_0)$, which ia also a graph in $\mathscr{B}(n,m,k_0)$

Suppose that $B^*\in \mathscr{B}(n,m,k_0-1)$. Then $\sigma_2(B^*)$
attains the maximum value among all the graphs in
$\mathscr{B}(n,m,k_0-1)$. Note that $m=(n-k_0)(k_0-1)$. Then we have
$(k_0-1)(n-k_0+1)-m=k_0-1$. It follows from Lemma 4 that
$\sigma_2(\overline {B^*})$ attains the maximum value among all the
graphs in $\mathscr{B}(n,k_0-1,k_0-1)$. By Theorem 1, we obtain that
$\overline {B^*}\cong K_{1,k_0-1}\cup S_{n-k_0}$. Since $k_0-1\geq
\lceil\frac{n}{2}\rceil$, we have $k_0-1\geq n-k_0+1$. So all the
pendent vertices are in $X^*$. By Lemma 4, we have  $B^\ast\cong
B^l(n,m,k_0-1)$.

\begin{subcase}
$f(k_0)>f(k_0-1)$.
\end{subcase}

In this case, we have $B^*\in \mathscr{B}(n,m,k_0)$ and
$\sigma_2(B^*)$ attains the maximum value among all the graphs in
$\mathscr{B}(n,m,k_0)$. From Claim 3 we know that
$\lfloor\frac{m}{k_0-1}\rfloor-\lfloor\frac{m}{k_0}\rfloor\leq 1$.
Suppose
$\lfloor\frac{m}{k_0-1}\rfloor-\lfloor\frac{m}{k_0}\rfloor=0$. Then
we have $q_0''=q_0$ and  $r_0''=q_0+r_0$. By $r_0''<k_0-1$, we get
$r_0<k_0-q_0-1$. If $r_0=0$, then
$$
m=(n-k_0)k_0>(n-k_0)(k_0-1).
$$
If $r_0>0$, then
\begin{align*}
m&=(n-k_0-1)k_0+r_0\\
&=(n-k_0)(k_0-1)+n-2k_0+r_0\\
&<(n-k_0)(k_0-1).
\end{align*}
Suppose
$\lfloor\frac{m}{k_0-1}\rfloor-\lfloor\frac{m}{k_0}\rfloor=1$.
Then
$m-\lfloor\frac{m}{k_0-1}\rfloor(k_0-1)>0$. Since
$\lfloor\frac{m}{k_0-1}\rfloor\geq n-k_0$, we have
$$
m>(n-k_0)(k_0-1).
$$
Therefore, in the following we consider two subcases.

\begin{subsubcase}
$m>(n-k_0)(k_0-1)$.
\end{subsubcase}

By Lemma 4 we know that $\sigma_2(\overline {B^*})$ attains the
maximum value among all the graphs in
$\mathscr{B}(n,k_0(n-k_0)-m,k_0)$. Since $k_0(n-k_0)-m<n-k_0\leq
n-1$, it follows from Theorem 1 that $B^*\cong
k_{1,k_0(n-k_0)-m}\cup S_{n-k_0(n-k_0)+m-1}$. If the $k_0(n-k_0)-m$
pendent vertices are in $X^\ast$, then by Lemma 4, we have $B^*\cong
B^l(n,m,k_0)$. If the $k_0(n-k_0)-m$ pendent vertices are in
$Y^\ast$, then by Lemma 4, we have $B^\ast\cong B^l(n,m,n-k_0)$,
which ia also a graph in $\mathscr{B}(n,m,k_0)$.

\begin{subsubcase}
$m<(n-k_0)(k_0-1)$.
\end{subsubcase}

It follows from Lemma 4 that $\sigma_2(\overline {B^*})$ attains the
maximum value among all the graphs in
$\mathscr{B}(n,k_0(n-k_0)-m,k_0)$. By Claim 1, we can conclude that
$k_0(n-k_0)-m<2k_0-n+1\leq n-1$. Then by Theorem 1 we have $B^*\cong
k_{1,k_0(n-k_0)-m}\cup S_{n-k_0(n-k_0)+m-1}$. Since
$k_0(n-k_0)-m>n-k_0$, we know that the $k_0(n-k_0)-m$ pendent
vertices are in $X^\ast$. By Lemma 4, we have $B^\ast\cong
B^l(n,m,k_0)$.

The proof is complete.
\end{proof}

%

\end{document}